\documentclass[12pt]{amsart}
 \usepackage{amsmath,amssymb,amsfonts,amsthm,amscd,textcomp,times,booktabs}
 \usepackage{mathrsfs}
 \usepackage{adjustbox, enumerate, amsbsy, stmaryrd}
\usepackage[dvipsnames]{xcolor}
\usepackage{amsmath,wasysym}
\usepackage{comment}
\usepackage{caption}
\usepackage{stackrel}
\usepackage{graphicx}
\usepackage {epic}
\usepackage[all,color]{xy}
\usepackage{extarrows}  
\ifx\pdfoutput\undefined
\fi
\usepackage{epstopdf}
\usepackage{braket}
\usepackage{color,float}
\usepackage{lscape}
\usepackage[alphabetic]{amsrefs}
\usepackage{tikz}
\usepackage{tikz-cd}
\usepackage{tikz-3dplot}
\newtheorem{theorem}{Theorem}

\newtheorem{claim}[theorem]{Claim}
\newtheorem{proposition}[theorem]{Proposition}

\newtheorem{definition}[theorem]{Definition}

\newtheorem{remark}[theorem]{Remark}

\numberwithin{equation}{section} 
\numberwithin{theorem}{section}

\newcommand{\mapright}[1]{\smash{\mathop{   \hbox to 0.7cm{\rightarrowfill}}
 \limits^{#1}}}

\newcommand{\Z}{\ensuremath{\mathbb{Z}}}
\newcommand{\Q}{\ensuremath{\mathbb{Q}}}
\newcommand{\R}{\ensuremath{\mathbb{R}}}
\newcommand{\C}{\ensuremath{\mathbb{C}}}

\newcommand{\GL}{\ensuremath{\mathrm{GL}}}

\newcommand{\reduce}[1]{\scalebox{1}{\ensuremath{#1}}}
\newcommand{\vol}{\ensuremath{\mathrm{vol}}}
\newcommand{\Vol}{\ensuremath{\mathrm{Vol}}}

\newcommand{\norm}[1]{\ensuremath{\left| #1 \right|}}

\newcommand{\restrict}[2]{\ensuremath{\left. #1 \right|_{#2}}}

 \DeclareMathOperator{\Aut}{Aut}

\DeclareMathOperator*{\smallcup}{\reduce{\bigcup}}

\newcommand{\Chow}{{\mathrm{Chow}}}

\def\Bl{\mathrm{Bl}}

\def\bs{\boldsymbol}
\def\C{\mathbb C}
\def\conv{{\rm{Conv}}}

\def\D{\Delta}

\def\p{\partial}

\def\R{{\mathbb R}}
\def\RR{{\mathbb R}}

\begin{document}
 \title[Blow-up formula of Chow stability]{On the blow-up formula of the Chow weights for polarized toric manifolds}

\author{King Leung Lee}
\address{IMAG, 
Universit\'{e} de Montpellier, $499$-$554$ Rue du Truel $9$, $34090$ Montpellier, France}
\email{king-leung.lee@umontpellier.fr}

\author{Naoto Yotsutani}

\address{Department of Mathematics, Faculty of Science, Shizuoka University, $836$ Ohya, Suruga-ku, Shizuoka-shi, Shizuoka, $422$-$8529$, Japan}
\email{yotsutani.naoto@shizuoka.ac.jp}

\address{IMAG, 
Universit\'{e} de Montpellier, $499$-$554$ Rue du Truel $9$, $34090$ Montpellier, France}
\email{naoto.yotsutani@umontpellier.fr}

\makeatletter
\@namedef{subjclassname@2020}{%
  \textup{2020} Mathematics Subject Classification}
\makeatother


\date{ \today}
\maketitle

\noindent{\bfseries Abstract.}
Let $X$ be a smooth projective toric variety, and let $\widetilde{X}$ denote the blow-up of $X$ at finitely many distinct torus-invariant points.
In this paper, we derive  an explicit combinatorial formula for the Chow weight of $\widetilde{X}$ in terms of the underlying toric manifold $X$ and the symplectic cuts of its associated Delzant polytope. As an application, we study toric blow-ups of the projective plane and compare their Chow stability with that of blow-ups at general points.

\section{Introduction}
Throughout this paper, we follow the notation and terminology of \cite{LY24}. 
Let $(X,L)$ be an $n$-dimensional polarized manifold, where $L\to X$ is an ample line bundle over an $n$-dimensional compact K\"ahler manifold.
Let $\Aut(X,L)$ denote the group of holomorphic automorphisms of $(X,L)$ modulo the trivial $\C^\times$-action.
It is well known, following Donaldson \cite{Don01}, that if a polarized manifold $(X,L)$ admits a constant scalar curvature K\"ahler (cscK) metric in the class $c_1(L)$,
and if the automorphism group $\Aut(X,L)$ is discrete, then $(X,L)$ is asymptotically Chow stable.
However, subsequent works \cite{OSY12,DVZ12,LLSW19} showed that this implication fails in the presence of non-trivial automorphisms.
In particular, Della Vedova and Zuddas \cite{DVZ12} used blow-up constructions to produce explicit examples of cscK classes that are asymptotically Chow unstable.
Independently, Odaka \cite{Oda12} asymptotically Chow unstable cscK orbifolds with trivial automorphism groups by developing a blow-up formula for the Donaldson-Futaki (DF) invariant along flag ideals.

To briefly recall the key ideas of \cite{DVZ12}, Della Vedova and Zuddas observed that the obstruction to asymptotic Chow semistability 
does not vanish under blow-ups at distinct points. 
Based on this insight, they established:
\begin{enumerate}
\item a blow-up formula for the obstruction to asymptotic Chow semistability; and
\item the existence of cscK metrics on appropriately polarized blow-up manifolds $(\widetilde{X},\widetilde{L})$, as ensured by the Arezzo-Pacard theorem \cite{AP09}.
\end{enumerate}
As a consequence, they constructed many new examples of asymptotically Chow unstable cscK classes in complex dimension two.
These developments highlight the importance of blow-up operations in understanding the relationship between GIT-stability and the existence of canonical K\"ahler metrics. We also remark that the DF invariant of a blow-up manifold was formulated and studied by Stoppa in \cite{St10}.

The goal of the present paper is to derive an explicit formula for the {\emph{Chow weight}} of a toric blow-up manifold using the technique of {\emph{symplectic cutting}}
\cite{Ler95, LT97}.
To state our main result more precisely, let $\D\subset \R^n$ be an $n$-dimensional Delzant lattice polytope.
Denote by $E_\D(t)$ the {\emph{Ehrhart polynomial}} of $\Delta$, which admits the asymptotic expansion 
\[
E_\D(t)=\Vol(\D)t^n+\frac{\Vol(\p \D)}{2}t^{n-1}+O(t^{n-2}).
\]
Recall that a {\emph{piecewise linear convex function}} $g(\bs x)$ on $\D$ is of the form
\[
g(\bs x)=\max\Set{g_1, \dots, g_\ell},
\]
where each $g_j$ is an affine function on $\D$.
The function $g$ is called {\emph{rational}} if 
\[
g_j(\bs x)=\sum_{i=1}^n a_{j,i}x_i+c_j, \qquad j=1, \dots, \ell,
\]
with some $(a_{j,1}, \dots, a_{j,n})\in \Q^n$ and $c_j\in \Q$. 

In \cite[Lemma 4.2]{YZ19}, the Chow weight of the associated polarized toric manifold $(X_\D, \mathcal O_{X_\D}(i))$ was explicitly computed following the arguments of \cite{RT07,ZZ08}. More precisely, consider the toric test configuration $(\mathscr X_f, \mathscr L_f)$ for $(X_\D, \mathcal O_{X_\D}(i))$ induced by a rational piecewise linear {\emph{concave}} function $f(\bs x)=-g(\bs x)$ on $\D$.
Then the Chow weight of $(\mathscr X_f, \mathscr L_f)$ is given by
\begin{equation}\label{def:ChowWt0}
\Chow_\D(f;i)=\Vol(\D)\sum_{\bs a\in \D\cap(\Z/i)^n}f(\bs a)-E_\D(i)\int_\D f(\bs x)\, dv.
\end{equation}
From the viewpoint of symplectic geometry,
let $(X_\D, \omega_\D, \mathbb T^n, \mu_\D)$ denote the associated symplectic toric manifold,
where $\mathbb T^n=(S^1)^n$ and 
\[
\mu_\D:X_\D \to \mathrm{Lie}(\mathbb T^n)^*\cong\R^n
\] 
is the corresponding moment map. The {\emph{symplectic blow-up}} of $(X_\D,\omega_\D)$ at a fixed point of the $\mathbb T^n$-action again yields a symplectic toric manifold.
More precisely, let $p$ be a $\mathbb T^n$-fixed point and let
\[
\bs v=\mu_\D(p)
\]
be the corresponding vertex of $\D$. For the symplectic blow-up
\[
\varpi: X_{\D^*}:=\mathrm{Bl}_p(X_\D)\dasharrow X_\D,
\]
it is known that the moment polytope $\D^*$ of the blow-up manifold $(X_{\D^*}, \omega_{\D^*})$ is obtained from $\D$ by cutting off the corner at $\bs v \in \mathcal V(\D)$ 
(cf. \cite[Chapter $29$]{CdS08}), where $\mathcal V(\D)$ denotes the set of vertices of $\D$.

More generally, suppose we perform symplectic blow-ups of $(X_\D, \omega_\D)$ at $\mathbb T^n$-fixed points $p_1, p_2, \ldots, p_\ell$. The resulting manifold
\[
X_{\D'}:=\mathrm{Bl}_{(p_1, \dots, p_\ell)}(X_\D) \dasharrow X_\D,
\]
has a moment polytope $\D'$, and the original polytope admits a decomposition
\begin{equation}\label{eq:decomp0}
\D=\D'\cup \smallcup_{a=1}^\ell D_a,
\end{equation}
where each $D_a$ is an $n$-dimensional simplex containing exactly one vertex of $\D$. (See Section \ref{sec:SetUp} for the precise terminology). 

Our main theorem concerns the Chow weight of toric blow-ups in complex dimension two: 
\begin{theorem}[See, Theorem \ref{thm:BlowUp}]\label{thm:main}
Let $\D\subset \R^2$ be a two-dimensional Delzant lattice polytope decomposed as in $\eqref{eq:decomp0}$, corresponding to the symplectic toric blow-up of $X_\D$ at  the torus-fixed points $p_1, \dots, p_\ell$. 
Assume that there exists the minimum integer $k\in \Z_{>0}$ such that $k\D'$ is a two-dimensional Delzant lattice polytope (see, Section \ref{sec:SetUp} for further details).
Let $\widetilde X:=X_{k\D'}$ be the associated toric surface. 
Then, for any $i\in \Z_{>0}$, the Chow weight of $\widetilde X$ with respect to the function $f(\bs x)=\bs x$ satisfies
\[
\Chow_{k\D'}(\bs x;i)= \Chow_{k\D}(\bs x;i)+i\mathcal{DF}_{k\D,1}+\mathcal{DF}_{k\D,2},
\]
where $\mathcal{DF}_{k\D,1}$ and $\mathcal{DF}_{k\D,2}$ are explicitly defined combinatorial invariants  (see $\eqref{eq:DF1}$ and $\eqref{eq:DF2}$).
\end{theorem}
See Proposition \ref{prop:BlowUp} for higher dimensional generalization. 
As an application, we study the blow-up of $\C P^2$  at six torus-fixed points. Let
\begin{itemize}
\item $Y$ denote the del Pezzo surface of degree $6$, written $dP_6$, obtained from $\C P^2$ by toric blow-ups at the three torus-fixed points of the standard $\mathbb T^2$-action; and  
\item $Z$ denote the complex surface obtained by further blowing up $Y$ at three additional torus-fixed points.
\end{itemize} 
We show that $Z:=\mathrm{Bl}_{(\text{$6$pts})}(\C P^2)$, constructed by this (toric) procedure, is {\emph{Chow unstable}}.

In contrast, let $X:=dP_3$ be the {\emph{del Pezzo surface of degree $3$}}, obtained by blowing up six general points in $\C P^2$.
Then $(X, -K_X)$ is known to be asymptotically Chow stable by the results of Tian \cite{Ti90} and Donaldson \cite{Don01}.
This contrast illustrates the subtle influence of {\emph{blow-up configuration}} - toric versus generic - on stability (see Section \ref{sec:6pts}).   

We note that the asymptotic behavior of a certain combinatorial invariant, called the \emph{$k$-quantized barycenter} in \cite{JR25}, has been studied by Jin and Rubinstein. In the same work, they established asymptotic expansions, with respect to the tensor powers
of a line bundle for algebro-geometric invariants, arising in K\"ahler geometry.
In addition, a related investigations of the Chow weights of blow-up varieties
\[
\varpi:X'=\mathrm{Bl}_Z(X)\dasharrow X,
\]
where $X$ is an irreducible normal projective variety and $Z\subsetneq X$ is a proper subvariety, were carried out in \cite{Gr23}.
While \cite{Gr23} approaches the problem from a purely algebro-geometric viewpoint, 
the emphasis of the present paper is on the combinatorial aspects of projective toric varieties.

This paper is organized as follows. In Section \ref{sec:Prelim}, we briefly review the notation and terminology used throughout the paper. We also introduce the notion of the {{Chow weight}} of a toric manifold in terms of its associated Delzant polytope $\D$, 
and establish its fundamental properties under affine transformations of $\D$ (see Proposition \ref{prop:AffineTrans}).
In Section \ref{sec:BlowUp}, we derive an explicit 
blow-up formula for the Chow weight of polarized toric manifolds and prove 
Theorem \ref{thm:main} in Section \ref{sec:ToricSurf}. 
Section \ref{sec:App} is devoted to applications of Theorem \ref{thm:main}, including explicit examples illustrating the geometric differences between toric blow-ups and general blow-ups. In particular, in Section \ref{sec:1stStep}, we derive an explicit formula for the Chow weights of {\emph{all}} toric surfaces whose moment polytopes have five vertices.

\vskip 7pt

\noindent{\bfseries Acknowledgements.}
The authors would like to thank the anonymous referee for valuable comments and suggestions that improved the manuscript.
The second author (NY) is especially grateful to Professor Hajime Ono for inspiring discussions on this subject and for generously sharing unpublished work during NY's  postdoctoral stay at USTC in China.
This work was supported by the first author's ANR-$21$-CE$40$-$0011$ JCJC project MARGE. 
The second author was partially supported by JSPS KAKENHI Grants JP$22$K$03316$, JP$24$KK$0252$, and JP$26$K$06777$. 

\section{Transformation and multiplication rules of combinatorial invariants of Delzant polytopes}\label{sec:Prelim}
\subsection{Notation and preliminary material}
Let $M$ be a free abelian group of rank $n$, so that $M\cong \Z^n$.
Let $\D$ be an $n$-dimensional lattice Delzant polytope in $M_\R:=M\otimes_\Z \R\cong \R^n$.
Denote by $E_\D(t)$ the Ehrhart polynomial of $\D$, which admits the expansion 
\begin{equation}\label{eq:Ehrhart}
E_\D(t)=\Vol(\D)t^n+\frac{\Vol(\p \Delta)}{2}t^{n-1}+O(t^{n-2}) \qquad \text{and} \qquad
E_\D(i)=\#(i\D\cap \Z^n)=\#(\D\cap (\Z/i)^n)
\end{equation}
for every positive integer $i$, where $\p\D$ denotes the boundary of $\D$. In particular, $E_\D(i)$ counts the number of lattice points in the $i$-th dilation of $\D$.

Analogously to the Ehrhart polynomial, it is known that there exists an $\R^n$-valued polynomial satisfying (with asymptotic expansion)
\begin{align}
\begin{split}\label{def:SumPoly}
\bs s_\D(i)&:=\sum_{\bs a\in \D \cap(\Z/i)^n}\bs a=\frac{1}{i}\sum_{\bs a\in i\D\cap \Z^n}\bs a\\
&=i^n\int_\D \bs x\, dv+\frac{i^{n-1}}{2}\int_{\p \D}\bs x\, d\sigma+\dots +\bs s_\D
\end{split}
\end{align}
for every positive integer $i$. Here $d\sigma$ denotes the $(n-1)$-dimensional Lebesgue measure on $\p\D$, defined as follows.
Let $h_j(\bs x)=\braket{\bs x, \bs v_j}+a_j$ be the defining affine function of the facet 
\[
F_j=\set{\bs x\in \D| h_j(\bs x)=0}.
\] 
Let $dv=dx_1\wedge \dots \wedge dx_n$ be the standard volume form on $\RR^n$. On each facet $F_j\subset \partial \Delta$, 
the induced $(n-1)$-dimensional Lebesgue measure $d\sigma_j=d\sigma|_{F_j}$ is defined by
\[
dv=\pm d\sigma_j\wedge dh_j.
\]
We call $\bs s_\D(t)$ the {\emph{lattice points sum polynomial}}.

\subsection{Chow weights}
Let $(X,L)$ be an $n$-dimensional polarized variety.
Recall that a \emph{test configuration} for $(X,L)$ consists of a polarized scheme $(\mathscr X,\mathscr L)$ together with the following data:
\begin{itemize}
\item a $\C^\times$-action and a proper flat morphism $\pi:\mathscr X\to \C$ that is $\C^\times$-equivariant
with respect to the standard action on $\C$;
\item a $\C^\times$-equivariant line bundle $\mathscr L\to \mathscr X$ that is ample over the fibers $\mathscr X_z:=\pi^{-1}(z)$, such that $(\mathscr X_z,  \mathscr L_z)\cong (X,L)$ for $z\neq 0$, where $ \mathscr L_z:= \restrict{\mathscr L}{\mathscr X_z}$.  
\end{itemize}
For any positive integer $i$, let $d_i=\dim H^0(X,L^{\otimes i})$. It was shown in \cite{RT07} that the data of a test configuration for $(X,L^{\otimes i})$ corresponds bijectively to a $\C^\times$-action on $\GL(d_i,\C)$.
For any $r\in \Z_{>0}$, set $k=ri$.
Let $w(k)=w(ri)$ denote the total weight of the induced $\C^\times$-action on $H^0(\mathscr X_0, \mathscr L_0^{\otimes r})$, where $(\mathscr X_0, \mathscr L_0)$ is the central fiber of the test configuration $(\mathscr X, \mathscr L)$. Define the normalized weight 
\[
\widetilde{w}_{i,k}:=w(k)id_i-w(i)kd_k.
\]
For sufficiently large $k$, it is known that $\widetilde{w}_{i,k}$ is a polynomial of degree $n+1$ in $k$, of the form 
\begin{equation}\label{eq:TotalWt}
\widetilde{w}_{i,k}=\sum_{j=1}^{n+1}e_j(i)k^j.
\end{equation}
The leading coefficient $e_{n+1}(i)$  is called the \emph{Chow weight}.

Now suppose that $(X,L)$ is an $n$-dimensional polarized toric variety with the associated polytope $\D$.
Let $f(\bs x)$ be a rational piecewise linear concave function on $\D$.
As shown in \cite[Lemma 4.2]{YZ19}, one can explicitly compute the Chow weight of the toric test configuration for $(X,L^{\otimes i})$ induced by $f$; it is given by $\eqref{def:ChowWt0}$.

Thus, we obtain the following combinatorial characterization of toric Chow weight.

\begin{definition}\label{def:ChWt}\rm
Let $\D\subset M_\R\cong\R^n$ be an $n$-dimensional Delzant lattice polytope, and let $f(\bs x)$ be a concave piecewise linear function on $\D$. For a positive integer $i$, define
\[
P_\D(f;i):=\sum_{\bs a\in \D\cap (\Z/i)^n}f(\bs a).
\]
We define the {\emph{Chow weight}} of $i\D$ with respect to $f$ by
\begin{equation}\label{def:ChowWt}
\Chow_\D(f;i):=\Vol(\D)P_\D(f;i)-E_\D(i) \int_\D f(\bs x) dv.
\end{equation}
\end{definition}
As a special case, if we take $f(\bs x)=\bs x$, the coordinate function, 
then $\eqref{def:ChowWt}$ becomes
\begin{equation}\label{def:ChowWt2}
\Chow_\D(\bs x;i)=\Vol(\D)\bs s_\D(i)-E_\D(i) \int_\D \bs x \,dv,
\end{equation}
which  coincides with {\emph{Ono's obstruction}} for Chow semistability for the polarized toric manifold $(X_\D,\mathcal O_{X_\D}(i))$ introduced in \cite[Theorem $1.2$]{Ono11}.

\begin{remark}\label{rem:ChowWt}\rm
\begin{enumerate}
\item For any $\mathbb{R}^n$-valued linear function $f$, $\Chow_\D(f;i)$ is an $\R^n$-valued polynomial in $i$ of degree $n-1$.
\item If the polarized toric manifold $(X_\D,L_\D)$ with $L_\D=\mathcal O_{X_\D}(1)$ is {\emph{asymptotically Chow semistable}}, then 
\[
\Chow_\D(f; i)\equiv \bs 0
\] 
for every positive integer $i$ and every affine linear function $f(\bs x)$.
\item As mentioned in \cite[$(1.3)$]{LY24}, the coefficients of $\Chow_\D(f;i)$ in $\eqref{def:ChowWt2}$ span the same linear space as
$\Set{\restrict{\mathcal F_{\mathrm{Td}^p}}{{\C^n}}: ~ p=1, \dots , n}$, where $\mathcal F_{\mathrm{Td}^p}$ denotes the family of integral invariants introduced in \cite{Fut04}.
\end{enumerate}
\end{remark}
In the following proposition, we assume that $f(\bs x)=A\bs x$ with $A\in \GL_n(\R)$, is a linear function on $\D$ with no constant term.
Then the Chow weights defined in $\eqref{def:ChowWt}$ satisfy the following natural transformation properties under affine transformations of $\D$.
\begin{proposition}\label{prop:AffineTrans}
Let $\D$ be an $n$-dimensional Delzant lattice polytope in $\R^n$. Let $i$ be any positive integer, and let $f(\bs x)=A\bs x$. 
\begin{enumerate}
\item ({\bf{Translation invariance}}) For any $\bs c\in \Z^n$, under the translation $\D+\bs c$, we have
\[
\Chow_{\D+\bs c}(f;i)=\Chow_\D(f;i).
\]
\item ({\bf{Unimodular transformation}}) For any $A'\in SL_n(\Z)$, under the transformation $A'\Delta$, we have
\[
\Chow_{A'\D}(f;i)=A'\Chow_\D(f;i).
\]
\item ({\bf{Scaling property}}) For any $m\in \Z_{>0}$, under the dilation $m\D$, we have
\[
\Chow_{m\D}(f;i)=m^{n+1}\Chow_\D(f;mi).
\]
\end{enumerate}
\end{proposition}
\begin{proof}
(1). For $\bs c\in\Z^n$, we observe that
\begin{align*}
E_{\D+\bs c}(i)&=E_\D(i), \qquad \Vol(\D+\bs c)=\Vol(\D), \\
P_{\D+\bs c}(f;i)&=\sum_{\bs a-\bs c \in \D \cap(\Z/i)^n}f(\bs a) =\sum_{\bs b \in \D \cap(\Z/i)^n}f(\bs b+\bs c)\\
&=\sum_{\bs b \in \D \cap(\Z/i)^n}f(\bs b)+\sum_{\bs c  \in \D \cap(\Z/i)^n}\bs c =P_\D(f;i)+E_\D(i)\bs c, \qquad \text{and}\\
\int_{\D+\bs c} f(\bs x) dv&=\int_\D f(\bs x) dv+\bs c\int_\D1\,dv=\int_\D f(\bs x) dv+\Vol(\D)\bs c.
\end{align*}
These yield 
\begin{align*}
\Chow_{\D+\bs c}(f;i)&=\Vol(\D+\bs c)\biggl(P_\D(f;i)+E_\D(i)\bs c\, \biggr)-E_{\D+\bs c}(i)\Set{\int_\D f(\bs x) dv+\Vol(\D)\bs c}\\
&=\Vol(\D)P_\D(f;i)-E_\D(i) \int_\D f(\bs x) dv=\Chow_\D(f;i).
\end{align*}

\noindent (2). For $A'\in SL_n(\Z)$, we see that
\begin{align*}
E_{A'\D}(i)&=\#(i(A'\D)\cap \Z^n)=\#(A'(i\D\cap \Z^n))=E_\D(i),\\
\Vol(A'\D)&=\norm{\det A'}\cdot \Vol(\D)=\Vol(\D),\\
P_{A'\D}(f;i)&=\sum_{\bs a \in A'\D\cap (\Z/i)^n}f(\bs a)=\sum_{\bs a \in A'(\D\cap (\Z/i)^n)}f(\bs a)=A'\sum_{\bs a \in \D\cap (\Z/i)^n}f(\bs a)
=A'\cdot P_\D(f;i),\\
\int_{A'\D}f(\bs x) dv&=\norm{\det A'}\cdot A'\int_\D f(\bs x)dv=A'\int_\D f(\bs x) dv, 
\end{align*} 
because $\norm{\det A'}=1$.
Therefore, we find 
\begin{align*}
\Chow_{A'\D}(f;i)&=\Vol(A'\D)A'\cdot P_\D(f;i)-E_{A'\D}(i)\cdot A'\int_\D f(\bs x)dv
=A'\Chow_\D(f;i).
\end{align*}

\noindent (3). For $m\in \Z_{>0}$, we have
\begin{align*}
E_{m\D}(i)&=\#(i(m\D)\cap \Z^n)=E_\D(mi),  \qquad \Vol(m\D)=m^n\Vol(\D), \\
P_{m\D}(f;i)&=\sum_{\bs a\in m\D\cap (\Z/i)^n}f(\bs a)=\sum_{\bs a\in m(\D\cap (\Z/mi)^n)}f(\bs a)=m\!\!\!\sum_{\bs a\in \D\cap (\Z/mi)^n}f(\bs a)=mP_\D(f;mi),\\
\int_{m\D}f(\bs x)  dv&=\int_\D f(m\bs x)m^ndv=m^{n+1}\int_\D f(\bs x) dv
\end{align*}
by linearity of $f(\bs x)$. Hence, we conclude 
\begin{align*}
\Chow_{m\D}(f;i)&=m^n\Vol(\D)\cdot mP_\D(f;mi)-E_\D(mi)\cdot m^{n+1}\int_\D f(\bs x)dv\\
&=m^{n+1}\Chow_\D(f;mi).
\end{align*}
Thus, the assertions are verified.
\end{proof}

\begin{remark}\rm
We note that for a general affine linear function $f(\bs x)=A\bs x+\bs b$, the Chow weight $\Chow(f;i)$ does not satisfy the properties stated in Proposition \ref{prop:AffineTrans}.
For example, the equality $f(m\bs x)=mf(\bs x)$, which is used in the proof of Proposition \ref{prop:AffineTrans} (3), fails when
$f(\bs x)=A\bs x+\bs b$. Moreover, Proposition \ref{prop:F1} below shows that $P_\D(f;i)$ does not admit a satisfactory asymptotic expression for a general affine function, whereas the situation behaves much better in the special vector valued case $f(\bs x)=\bs x$.
\end{remark}
\begin{proposition}\label{prop:F1}
Let
\[
\D=\conv\Set{
\begin{pmatrix}
0 \\ 0
\end{pmatrix},
\begin{pmatrix}
2 \\ 0
\end{pmatrix},
\begin{pmatrix}
1 \\ 1
\end{pmatrix},
\begin{pmatrix}
0 \\ 1
\end{pmatrix}
},
\]
and let $f(\bs x)=A \bs x+\bs b$ be an affine function on $\D$, where $A=\begin{pmatrix}
1 & 2 \\
2 & 1
\end{pmatrix}$ and $\bs b=\begin{pmatrix}
1 \\ 1
\end{pmatrix}$. Then $P_\D(f;i)$ cannot be expressed in the form 
\begin{equation}\label{eq:EM}
P_\D(f;i)=i^2\int_\D f(\bs x)dv+\frac{i}{2}\int_{\p\D}f(\bs x)d\sigma+
\begin{pmatrix}
c_1 \\ c_2
\end{pmatrix}.
\end{equation}
\end{proposition}
\begin{proof}
Suppose that $P_\D(f;i)$ admits an expression of the form $\eqref{eq:EM}$.
By direct computation, we obtain
\begin{align*}\label{eq:values}
\int_\D f(\bs x)dv&=
\begin{pmatrix}
4 \\ \frac{9}{2}
\end{pmatrix}, \quad 
\int_{\partial \D} f(\bs x)d\sigma=
\begin{pmatrix}
13 \\ 15
\end{pmatrix}, 
\quad
\sum_{\bs a \in \D \cap \Z^2}f(\bs a)=
\begin{pmatrix}
13 \\ 15
\end{pmatrix}.
\end{align*}
Substituting these values into $\eqref{eq:EM}$ for $i=1$, we obtain $(c_1,c_2)=\left( \frac{5}{2}, 3\right)$.
Consequently, the right hand side of $\eqref{eq:EM}$ yields
$\left( \frac{63}{2}, 36\right)$ when $i=2$. 

However, direct computation shows that
\[
P_\D(f;2)=\frac{1}{2}\sum_{\bs a \in 2\D\cap \Z^2}f(\bs a)=\left( \frac{51}{2}, 30\right).
\]
This contradiction proves the proposition.
\end{proof}

\section{Blow-up formula of Chow weights of projective toric manifolds}\label{sec:BlowUp}
In this section, we derive an explicit formula for the Chow weights of the blow-ups of polarized toric manifolds.
Throughout Sections \ref{sec:BlowUp}--\ref{sec:App}, we fix a positive integer $i\in \Z_{>0}$, and regard it as the polarization parameter of $X_\D$, namely $(X_\D, \mathcal O_{X_\D}(i))$.
\subsection{Setting}\label{sec:SetUp}
Let $\D\subset M_\R\cong \R^n$ be an $n$-dimensional Delzant lattice polytope.
Suppose that $\D$ admits a decomposition 
\begin{equation}\label{eq:decomposition}
\D=\D'\cup \smallcup_{a=1}^\ell D_a
\end{equation}
satisfying the following properties:
\begin{itemize}
\item There exists $k\in \Z_{>0}$ such that $k\D'$ is an $n$-dimensional Delzant lattice polytope. We further assume that $k$ is the {\emph{minimal}} positive integer with this property. 
\item For each $a=1, \dots , \ell$, the set $D_a$ is an $n$-dimensional simplex containing a unique vertex of $\D$, which we denote by $\bs p_a$.
\item For distinct $a,b =1, \dots , \ell$, we have
\[
D_a\cap \D'=\p D_a\cap \D'=:L_a, \qquad \qquad D_a\cap D_b=\emptyset. 
\]
\end{itemize}
From the viewpoint of symplectic toric geometry, 
blowing up a toric symplectic manifold $X_\D$ at torus-fixed points corresponds to the operation of chopping off corners of the Delzant polytope $\D$. Consequently, the decomposition $\eqref{eq:decomposition}$ induces a toric morphism
\begin{equation}\label{eq:symp-blow-up}
X_{\D'}:=\mathrm{Bl}_{(p_1,\dots, p_\ell)}(X_\D)\dasharrow X_\D,
\end{equation}
where $p_1, \dots, p_\ell$ are the torus-fixed points of $X_\D$ corresponding to the vertices $\bs p_1,\dots, \bs p_\ell$ of $\D$.

\subsection{The Arbitrary-Dimensional Case}
We first establish a combinatorial description of the blow-up formula for Chow weights associated with the morphism $\eqref{eq:symp-blow-up}$ using the decomposition $\eqref{eq:decomposition}$.
For a function $f$ defined on $\D$ and a positive integer $i\in \Z_{>0}$, define
\[
P_{kL_a}(f;i):=\!\!\! \sum_{\bs a\in kL_a \cap (\Z/i)^n}f(\bs a), \qquad E_{kL_a}(i):=\#(kL_a \cap (\Z/i)^n).
\]
Using the decomposition $\eqref{eq:decomposition}$, the combinatorial invariants introduced in Section \ref{sec:Prelim} satisfy the following relations:
\begin{align}
\begin{split}\label{eq:Invariants}
\Vol(k\D)&=\Vol(k\D')+\sum_{a=1}^\ell\Vol(kD_a), \qquad \int_{k\D}f(\bs x) dv=\int_{k\D'}f(\bs x)dv+\sum_{a=1}^\ell\int_{kD_a}f(\bs x) dv, \\
P_{k\D}(f;i)&=P_{k\D'}(f;i)+\sum_{a=1}^\ell(P_{kD_a}(f;i)-P_{kL_a}(f;i)), \qquad \text{and}\\ 
E_{k\D}(i)&=E_{k\D'}(i)+\sum_{a=1}^\ell(E_{kD_a}(i)- E_{kL_a}(i)). 
\end{split}
\end{align} 
The following proposition gives the blow-up formula for Chow weights associated with the toric morphism $X_{\D'}\dasharrow X_\D$ in arbitrary dimension.
\begin{proposition}\label{prop:BlowUp}
Using the notation introduced in Section \ref{sec:SetUp}, we have
\begin{align*}
\Chow_{k\D'}(f;i)
&=\Chow_{k\D}(f;i)  \\
&-\left( \sum_{a=1}^\ell\Vol(kD_a) \right)P_{k\D}(f;i)-\left( \Vol(k\D)-\sum_{b=1}^\ell\Vol(kD_b) \right)\sum_{a=1}^\ell(P_{kD_a}(f;i)-P_{kL_a}(f;i))\\
&+E_{k\D}(i)\sum_{a=1}^\ell\int_{kD_a}f(\bs x) dv +\sum_{b=1}^\ell(E_{kD_b}(i)- E_{kL_b}(i))\left(  \int_{k\D}f(\bs x) dv -\sum_{a=1}^\ell\int_{kD_a}f(\bs x) dv \right).
\end{align*} 
\end{proposition}
\begin{proof}
Plugging each value of $\eqref{eq:Invariants}$ into $\eqref{def:ChowWt}$, we find that
\begin{align*}
\Chow_{k\D'}(f;i)&=\Vol(k\D')P_{k\D'}(f;i)-E_{k\D'}(i)\int_{k\D'}f(\bs x)dv\\
&=\left( \Vol(k\D)-\sum_{a=1}^\ell\Vol(kD_a) \right)\Set{P_{k\D}(f;i) -\sum_{a=1}^\ell(P_{kD_a}(f;i)-P_{kL_a}(f;i))}\\
&-\Set{E_{k\D}(i)-\sum_{a=1}^\ell(E_{kD_a}(i)- E_{kL_a}(i))}\left(  \int_{k\D}f(\bs x) dv -\sum_{a=1}^\ell\int_{kD_a}f(\bs x) dv \right) \\
&=\Chow_{k\D}(i) \\
&-\left( \sum_{a=1}^\ell\Vol(kD_a) \right)P_{k\D}(f;i)-\left( \Vol(k\D)-\sum_{b=1}^\ell\Vol(kD_b) \right)\sum_{a=1}^\ell(P_{kD_a}(f;i)-P_{kL_a}(f;i))\\
&+E_{k\D}(i)\sum_{a=1}^\ell\int_{kD_a}f(\bs x) dv +\sum_{b=1}^\ell(E_{kD_b}(i)- E_{kL_b}(i))\left(  \int_{k\D}f(\bs x) dv -\sum_{a=1}^\ell\int_{kD_a}f(\bs x) dv \right).
\end{align*} 
This completes the proof.
\end{proof}

\subsection{The Two-Dimensional Case}\label{sec:ToricSurf}
We now specialize to the case where $\Delta\subset \R^2$ is a two-dimensional Delzant polygon.
In this case, the Ehrhart polynomial introduced in $\eqref{eq:Ehrhart}$ is given by Pick's theorem:
\[
E_\D(i)=\Vol(\D)i^2+\frac{\#(\partial \D \cap \Z^2)}{2}i+1.
\]
For simplicity, throughout this subsection we only consider the vector-valued function $f(\bs x)=\bs x$.

Each two-dimensional simplex $D_a$ appearing in the decomposition $\eqref{eq:decomposition}$ can be described explicitly as follows.

 Let 
$\begin{pmatrix}
\alpha_a \\ \beta_a
\end{pmatrix}$, 
$\begin{pmatrix}
\gamma_a \\ \delta_a
\end{pmatrix} \in \Z^2$ be the primitive generators of the two edges of $\D$ emanating from the vertex $\bs p_a$. 
We orient them so that 
\begin{equation}\label{eq:Matrix}
\det A_a=1, \qquad A_a:=
\begin{pmatrix}
\alpha_a & \gamma_a \\
\beta_a & \delta_a
\end{pmatrix},
\end{equation}
which ensures that $\begin{pmatrix}
\alpha_a \\ \beta_a
\end{pmatrix}$ and
$\begin{pmatrix}
\gamma_a \\ \delta_a
\end{pmatrix}$ form an oriented basis compatible with the Delzant condition.

Since $k\D'$ is a Delzant lattice polygon by assumption (see, Section \ref{sec:SetUp}), there exists a positive integer $m_a\in \Z_{>0}$ such that
the chopped-off simplex $D_a$ can be written as
\begin{align*} 
D_a&=\D(\bs p_a,\bs q_a, \bs r_a):=\mathrm{conv}\set{\bs p_a,\bs q_a, \bs r_a}, 
\end{align*}
where
\[
\bs q_a=\bs p_a+\frac{m_a}{k}
\begin{pmatrix}
\alpha_a \\ \beta_a 
\end{pmatrix},
 \qquad  
\bs r_a=\bs p_a+\frac{m_a}{k}
\begin{pmatrix}
\gamma_a \\ \delta_a 
\end{pmatrix}.
\]
Let $\D_{m_a}\subset \R^2$ be the standard right triangle
\[
\mathrm{conv}\Set{\begin{pmatrix}
0 \\ 0
\end{pmatrix},
\begin{pmatrix}
m_a \\ 0
\end{pmatrix},
\begin{pmatrix}
0 \\ m_a
\end{pmatrix}
}.
\]
Then $D_a$ is obtained from $\D_{m_a}$ by translation and a unimodular transformation $A_a\in \mathrm{SL}_2(\Z)$, namely, 
\begin{equation}\label{eq:Simplex}
kD_a=A_a\D_{m_a}+k\bs p_a.
\end{equation}
Let $l_{m_a}$ denote the edge of $\D_{m_a}$ joining 
$\begin{pmatrix}
m_a \\ 0
\end{pmatrix}$ and $\begin{pmatrix}
0 \\ m_a
\end{pmatrix}$.
Straightforward computations yield the following identities.
\begin{claim}\label{claim:Rect}
\begin{align}
\Vol(\D_{m_a})&=\frac{m_a^2}{2}, \notag \\ \label{eq:SumPoly2}
\bs s_{\D_{m_a}}(i)-\bs s_{l_{m_a}}(i)&=\frac{m_a}{6}(im_a+1)(im_a-1)
\begin{pmatrix}
1 \\ 1
\end{pmatrix},
\\ \label{eq:Ehrhart2}
E_{\D_{m_a}}(i)-E_{l_{m_a}}(i)&=\frac{im_a}{2}(im_a+1) \qquad \text{ and } \quad \\
\int_{\D_{m_a}}\bs x\, dv &=\frac{m_a^3}{6} \begin{pmatrix}
1 \\ 1
\end{pmatrix}. \notag
\end{align}  
\end{claim}
\begin{proof}
It suffices to evaluate $\eqref{eq:SumPoly2}$ and $\eqref{eq:Ehrhart2}$. Since the triangle $\D_{m_a}$ is symmetric with respect to 
the interchange of the $x$- and $y$-coordinates, 
it is enough to compute only the $x$-component.

For each $k=0,1, \dots, im_a$, there are $im_a-k$ lattice points in $(i\D_{m_a}\setminus il_{m_a})\cap \Z^2$ whose $x$-coordinate equals $k$. Hence the $x$-component of the left-hand side of $\eqref{eq:SumPoly2}$ is 
\begin{align*}
\frac{1}{i}\sum_{k=0}^{im_a}k(im_a-k)&=\frac{1}{i}\left( im_a\sum_{k=0}^{im_a} k- \sum_{k=0}^{im_a} k^2 \right)\\
&=\frac{im_a^2}{2}(im_a+1)-\frac{m_a}{6}(im_a+1)(2im_a+1)\\
&=\frac{m_a}{6}(im_a+1)(3im_a-2im_a-1)=\frac{m_a}{6}(im_a+1)(im_a-1).
\end{align*}
Similarly, $\eqref{eq:Ehrhart2}$ is given by
\[
E_{\D_{m_a}}(i)-E_{l_{m_a}}(i)=\sum_{k=0}^{im_a}(im_a-k)=\sum_{j=1}^{im_a}j=\frac{im_a}{2}(im_a+1).
\]
This proves Claim \ref{claim:Rect}.
\end{proof}
Using Claim \ref{claim:Rect}, $\eqref{eq:Simplex}$ and the transformation rules discussed in Section \ref{sec:Prelim}, we find that
\begin{align}
\begin{split}\label{eq:Invariants2}
\Vol(kD_a)&=\frac{m_a^2}{2},\\
\bs s_{kD_a}(i)-\bs s_{kL_a}(i)&=\frac{m_a}{6}(im_a+1)(im_a-1)
\begin{pmatrix}
\alpha_a+\gamma_a \\
\beta_a+\delta_a
\end{pmatrix}
+\frac{im_a}{2}(im_a+1)k\bs p_a, \\
E_{kD_a}(i)-E_{kL_a}(i)&=\frac{im_a}{2}(im_a+1),\\
\int_{kD_a}\bs x\, dv&=\frac{m_a^3}{6}
\begin{pmatrix}
\alpha_a+\gamma_a \\
\beta_a+\delta_a
\end{pmatrix}
+\frac{m_a^2}{2}k\bs p_a.
\end{split}
\end{align}
Setting
\[
M:=\sum_{a=1}^\ell m_a, \quad \widetilde M:=\sum_{a=1}^\ell m_a^2, \quad  A:=\#(\p(k\D)\cap \Z^2)-M \quad \text{and} \quad B:=2\Vol(k\D)-\widetilde M,
\]
respectively, we define $\mathcal{DF}_{k\D,1}$ and $\mathcal{DF}_{k\D,2}$ by
\begin{align}
\begin{split}\label{eq:DF1}
\mathcal{DF}_{k\D,1}=&\frac{A}{12}\sum_{a=1}^\ell m_a^3
\begin{pmatrix}
\alpha_a+\gamma_a \\
\beta_a+\delta_a
\end{pmatrix}
+\frac{k}{4}\left(A\sum_{a=1}^\ell m_a^2 \bs p_a-B\sum_{a=1}^\ell m_a \bs p_a\right)\\
&\qquad +\frac{M}{2}\int_{k\D}\bs x\, dv-\frac{\widetilde M}{4}\int_{\p(k\D)}\bs x\, d\sigma, \qquad \text{and} 
\end{split}\\  \label{eq:DF2}
\mathcal{DF}_{k\D,2}=&\frac{B}{12}\sum_{a=1}^\ell m_a \begin{pmatrix}
\alpha_a+\gamma_a \\
\beta_a+\delta_a
\end{pmatrix}
+\frac{1}{6}\sum_{a=1}^\ell m_a^3
\begin{pmatrix}
\alpha_a+\gamma_a \\
\beta_a+\delta_a
\end{pmatrix}+\frac{k}{2}\sum_{a=1}^\ell m_a^2 \bs p_a -\frac{\widetilde M}{2}\bs s_{k\D}.
\end{align}
We are now ready to state the main theorem of this paper.
\begin{theorem}\label{thm:BlowUp}
Let $\D\subset \R^2$ be a lattice Delzant polygon admitting a decomposition
\[
\D=\D'\cup \smallcup_{a=1}^\ell D_a
\]
satisfying the assumptions of Section \ref{sec:SetUp}. Let $k\in \Z_{>0}$ be the minimal positive integer such that $k\D'$ is a Delzant lattice polygon,
and let $\varpi:X_{\D'} \dasharrow X_\D$ be the corresponding symplectic toric blow-up 
at the torus-fixed points $p_1,\dots, p_\ell$.
Then, for every positive integer $i\in \Z_{>0}$, 
the Chow weight of $\widetilde{X}:=X_{k\D'}$ with respect to the polarization $\mathcal O_{\widetilde X}(i)$ satisfies
\begin{equation}\label{eq:BlowUp}
\Chow_{k\D'}(\bs x; i)=\Chow_{k\D}(\bs x;i)+i\mathcal{DF}_{k\D,1}+\mathcal{DF}_{k\D,2},
\end{equation}
where $\mathcal{DF}_{k\D,1}$ and $\mathcal{DF}_{k\D,2}$ are the combinatorial invariants defined in $\eqref{eq:DF1}$ and $\eqref{eq:DF2}$, respectively.
\end{theorem}
\begin{proof}
Plugging each value in $\eqref{eq:Invariants2}$ into Proposition \ref{prop:BlowUp}, we find that
\begin{align}
\begin{split}\label{eq:ChowWt2}
\Chow_{k\D'}(\bs x; i)&=\Chow_{k\D}(\bs x; i)-\left(\sum_{a=1}^\ell \frac{m_a^2}{2}\right)\bs s_{k\D}(i)\\
&-\left(\Vol(k\D)-\sum_{b=1}^\ell\frac{m_b^2}{2} \right)\sum_{a=1}^\ell
\left( \frac{m_a}{6}(im_a+1)(im_a-1)
\begin{pmatrix}
\alpha_a+\gamma_a \\
\beta_a+\delta_a
\end{pmatrix}
+\frac{im_a}{2}(im_a+1)k\bs p_a
\right)\\
&+E_{k\D}(i)\sum_{a=1}^\ell
\Set{\frac{m_a^3}{6}
\begin{pmatrix}
\alpha_a+\gamma_a \\
\beta_a+\delta_a
\end{pmatrix}
+\frac{m_a^2}{2}k\bs p_a}\\
&+\sum_{b=1}^\ell\left(\frac{im_b}{2}(im_b+1)
\right)\left(
\int_{k\D}\bs x\, dv- \sum_{a=1}^\ell
\Set{\frac{m_a^3}{6}
\begin{pmatrix}
\alpha_a+\gamma_a \\
\beta_a+\delta_a
\end{pmatrix}
+\frac{m_a^2}{2}k\bs p_a}
\right).
\end{split}
\end{align}
In $\eqref{eq:ChowWt2}$, the last four terms must be simplified as $c_1i+c_0$ with appropriate coefficients $c_0$ and $c_1$ with respect to $i$ (see, Remark \ref{rem:ChowWt} (1)).
Hence, we further proceed our computations.

\vskip 7pt

\noindent $\bullet$ The coefficient of $\begin{pmatrix}
\alpha_a+\gamma_a \\
\beta_a+\delta_a
\end{pmatrix}$:  
\begin{align*}
\sum_{a=1}^\ell&\left[
-\frac{1}{6}\left(\Vol(k\D)-\sum_{b=1}^\ell\frac{m_b^2}{2} \right)m_a(im_a+1)(im_a-1)+\frac{E_{k\D}(i)m_a^3}{6} \right. \\
 &\left. \hspace{7cm} -\frac{m_a^3}{12}\sum_{b=1}^\ell im_b(im_b+1)  \right]
 \begin{pmatrix}
\alpha_a+\gamma_a \\
\beta_a+\delta_a
\end{pmatrix}\\
&=\sum_{a=1}^\ell\left[
-\frac{1}{6}\left(\Vol(k\D)-\sum_{b=1}^\ell\frac{m_b^2}{2} \right)(m_a^3i^2-m_a)+\left( \Vol(k\D)i^2+\frac{\#(\p(k\D)\cap \Z^2)}{2}i+1\right)\frac{m_a^3}{6} \right. \\
&\left. \hspace{7cm}  -\frac{m_a^3}{12}\left(i^2 \sum_{b=1}^\ell m_b^2+i \sum_{b=1}^\ell m_b \right) \right]
 \begin{pmatrix}
\alpha_a+\gamma_a \\
\beta_a+\delta_a
\end{pmatrix}\\
&=\sum_{a=1}^\ell\left[
\frac{m_a^3}{12}\left(\#(\p(k\D)\cap \Z^2)-\sum_{b=1}^\ell m_b \right)i +\frac{m_a}{12}\left(2\Vol(k\D)+2m_a^2-\sum_{b=1}^\ell m_b^2\right)
\right] \begin{pmatrix}
\alpha_a+\gamma_a \\
\beta_a+\delta_a
\end{pmatrix}.
\end{align*}

\vskip 7pt

\noindent $\bullet$ The coefficient of $k\bs p_a$:  
\begin{align*}
&\sum_{a=1}^\ell\left[
-\frac{1}{2}\left(\Vol(k\D)-\sum_{b=1}^\ell \frac{m_b^2}{2}
\right)im_a(im_a+1)+E_{k\D}(i)\frac{m_a^2}{2}-\frac{1}{4}\sum_{b=1}^\ell im_b(im_b+1)m_a^2\right]k\bs p_a\\
=&\sum_{a=1}^\ell\left[
-\frac{1}{2}\left(\Vol(k\D)-\sum_{b=1}^\ell \frac{m_b^2}{2}\right)
(m_a^2i^2+m_ai)+\frac{1}{2}\left( \Vol(k\D)i^2+\frac{\#(\p(k\D)\cap \Z^2)}{2}i+1\right)m_a^2 \right.\\
&\left. \hspace{10cm}  -\frac{m_a^2}{4}  \left(i^2 \sum_{b=1}^\ell m_b^2+i\sum_{b=1}^\ell m_b \right) \right]k\bs p_a\\
=&\sum_{a=1}^\ell \left[ \frac{m_a}{4}\left( -2\Vol(k\D)+\sum_{b=1}^\ell m_b^2+\#(\p(k\D)\cap \Z^2)m_a-m_a\sum_{b=1}^\ell m_b
\right)i+\frac{m_a^2}{2}\right]k\bs p_a.
\end{align*}

\vskip 7pt

\noindent $\bullet$ The remaining terms:  
\begin{align*}
&-\frac{1}{2}\left(\sum_{a=1}^\ell m_a^2\right) \bs s_{k\D}(i)+\frac{1}{2}\sum_{a=1}^\ell (m_a^2i^2+m_ai)\int_{k\D} \bs x\, dv\\
=&\frac{1}{2}\sum_{a=1}^\ell \Set{-m_a^2 \left(i^2\int_{k\D} \bs x\,dv +\frac{i}{2}\int_{\p(k\D)} \bs x\,d\sigma+ \bs s_{k\D}\right)
+(m_a^2i^2+m_ai)\int_{k\D}\bs x\, dv} \\
=&\frac{1}{4}\sum_{a=1}^\ell\Set{m_a \left(-m_a \int_{\p(k\D)} \bs x\,d\sigma+2\int_{k\D} \bs x\,dv \right)i-2m_a^2\bs s_{k\D}}.
\end{align*}

Plugging these values into $\eqref{eq:ChowWt2}$, we conclude that
\begin{align*}
\Chow_{k\D'}(i)=&\Chow_{k\D}(i)\\
&+i\sum_{a=1}^\ell\frac{m_a}{4}\left[\frac{m_a^2}{3}\left(\#(\p(k\D)\cap \Z^2)-\sum_{b=1}^\ell m_b\right)
\begin{pmatrix}
\alpha_a+\gamma_a \\
\beta_a+\delta_a
\end{pmatrix}  \right. \\
+&\left(-2\Vol(k\D)+\sum_{b=1}^\ell m_b^2+\#(\p(k\D)\cap\Z^2)m_a-m_a\sum_{b=1}^\ell m_b\right)k\bs p_a\\
-&\left. m_a\int_{\p(k\D)}\bs x\, d\sigma+2\int_{k\D}\bs x\,dv \right] \\
+&\sum_{a=1}^\ell \frac{m_a}{12}\left[\left( 2\Vol(k\D)+2m_a^2-\sum_{b=1}^\ell m_b^2 \right)
\begin{pmatrix}
\alpha_a+\gamma_a \\
\beta_a+\delta_a
\end{pmatrix}+6m_a(k\bs p_a-\bs s_{k\D}) \right]\\
=&\Chow_{k\D}(i)+i\, \mathcal{DF}_{k\D,1}+ \mathcal{DF}_{k\D,2}.
\end{align*}
Hence, the assertion is verified.
\end{proof}

\section{Application to toric surfaces}\label{sec:App}
For any Delzant lattice polygon $\D\subset \R^2$, we observe that
\begin{align}
\begin{split}\label{eq:Ehr_sum}
E_\D(i)&=\Vol(\D)i^2+(E_\D(1)-\Vol(\D)-1)i+1, \quad \text{and} \\
\bs s_\D(i)&=i^2\int_\D\bs x\, dv+i\left( \bs s_\D(2)-\bs s_\D(1)-3\int_\D \bs x\, dv \right)+2\int_\D \bs x\, dv-\bs s_\D(2)+2\bs s_\D(1).
\end{split}
\end{align}
In this section, we apply the blow-up formula established in Theorem \ref{thm:BlowUp} to the projective plane.

\subsection{The blow-up of $\C P^2$ at three distinct points}\label{sec:3pts}
Let $X$ denote the blow-up of $\C P^2$ at the three torus-fixed points 
\[
p_1=[1:0:0], \quad p_2=[0:1:0], \quad  p_3=[0:0:1].
\] 
It is well-known that $(X,\mathcal O_X(-K_X))$ is a smooth symmetric K\"ahler-Einstein toric del Pezzo surface with vanishing Futaki invariant.
In particular, by \cite[Corollary $1.5$]{LY24}, the polarized manifold $(X,\mathcal O_X(-K_X))$ is asymptotically Chow polystable;
see also \cite[Example $9.2$]{Lee25}.

We can also verify the Chow polystability of this example directly using the blow-up formula. 
Let
\begin{align}
\begin{split}\label{eq:polytope}
\D&=\conv \Set{
\begin{pmatrix}
0 \\ 0
\end{pmatrix}, 
\begin{pmatrix}
3 \\ 0
\end{pmatrix}, 
\begin{pmatrix}
0 \\ 3
\end{pmatrix}
}, \quad 
\D'=\conv \Set{
\begin{pmatrix}
1 \\ 0
\end{pmatrix}, 
\begin{pmatrix}
2 \\ 0
\end{pmatrix}, 
\begin{pmatrix}
2 \\ 1
\end{pmatrix},
\begin{pmatrix}
1 \\ 2
\end{pmatrix},
\begin{pmatrix}
0 \\ 2
\end{pmatrix},
\begin{pmatrix}
0 \\ 1
\end{pmatrix}
},\\
D_1&=\conv\Set{\begin{pmatrix}
0 \\ 0
\end{pmatrix},
\begin{pmatrix}
1 \\ 0
\end{pmatrix},
\begin{pmatrix}
0 \\ 1
\end{pmatrix}
}, \quad
D_2=\conv\Set{
\begin{pmatrix}
3 \\ 0
\end{pmatrix},
\begin{pmatrix}
2 \\ 1
\end{pmatrix},
\begin{pmatrix}
2 \\ 0
\end{pmatrix}
}, \qquad  \\
D_3&=\conv\Set{
\begin{pmatrix}
0 \\ 3
\end{pmatrix},
\begin{pmatrix}
0 \\ 2
\end{pmatrix},
\begin{pmatrix}
1 \\ 2
\end{pmatrix}
}.
\end{split}
\end{align}
Then, $\D$ will be decomposed into
\[
\D=\D'\cup \smallcup_{a=1}^3 D_a
\]
with $\bs p_1=(0,0)$, $\bs p_2=(3,0)$, $\bs p_3=(0,3)$. Since $\D'$ is already a lattice polygon, we may take $k=1$. 
Moreover, $m_a=1$ for each $a=1,2,3$, and the matrices defined in 
$\eqref{eq:Matrix}$ are 
\[
A_1=\begin{pmatrix}
1 & 0 \\
0 & 1
\end{pmatrix}, \qquad
A_2=\begin{pmatrix}
-1 & -1 \\
1 & 0
\end{pmatrix}, \qquad 
A_3=\begin{pmatrix}
0 & 1 \\
-1 & -1
\end{pmatrix}.
\]
Therefore, 
\begin{align*}
M&=\sum_{a=1}^3m_a=3, \quad  \widetilde M=\sum_{a=1}^3m_a^2=3,  \quad \text{ and } \\
 A&=\#(\p \D\cap \Z^2)-M=6, \qquad B=2\vol(\D)-\widetilde M=6.
\end{align*}
Since 
\[
\int_\D \bs x\, dv=\left(\frac{9}{2}, \, \frac{9}{2} \right),
\]
formula $\eqref{eq:Ehr_sum}$ gives
\[
\bs s_{\D}(i)=\frac{9}{2} 
\begin{pmatrix}
1 \\ 1
\end{pmatrix}i^2
+\frac{9}{2}
\begin{pmatrix}
1 \\ 1
\end{pmatrix}i+
\begin{pmatrix}
1 \\ 1
\end{pmatrix}.
\]
Substituting these values into $\eqref{eq:DF1}$ and $\eqref{eq:DF2}$, we obtain 
\begin{align*}
\mathcal{DF}_{\D, 1}&=\frac{27}{4}\left\{ \begin{pmatrix}
1 \\ 1
\end{pmatrix}-
\begin{pmatrix}
1 \\ 1
\end{pmatrix}
\right\}=
\begin{pmatrix}
0 \\ 0
\end{pmatrix},  \\
\mathcal{DF}_{\D, 2}&=\frac{1}{2}\sum_{a=1}^3p_a-\frac{3}{2}\bs s_\D=\frac{1}{2}\left\{ \begin{pmatrix}
0 \\ 0
\end{pmatrix}+
\begin{pmatrix}
3 \\ 0
\end{pmatrix}+
\begin{pmatrix}
0 \\ 3
\end{pmatrix}
\right\}-\frac{3}{2}\begin{pmatrix}
1 \\ 1
\end{pmatrix}=\begin{pmatrix}
0 \\ 0
\end{pmatrix}.
\end{align*}
Hence, 
\[
\Chow_{\D'}(\bs x;i)=\Chow_\D(\bs x; i)+i\begin{pmatrix}
0 \\ 0
\end{pmatrix}+\begin{pmatrix}
0 \\ 0
\end{pmatrix}=\begin{pmatrix}
0 \\ 0
\end{pmatrix}
\]
by $\eqref{eq:BlowUp}$, which is consistent with \cite[Corollary $1.5$]{LY24}.

\subsection{Blow-up of $\C P^2$ at four distinct points}\label{sec:4pts}
We next consider the blow-up of $\C P^2$ at four points. 
To obtain the {\emph{toric blow-up}} 
manifold $\mathrm{Bl}_{(\text{$4$pts})}(\C P^2)$, 
the set of blown-up points $\Lambda=\set{p_1,p_2,p_3,p_4}$ need to satisfy the condition
that $p_1, p_3$ and $p_4$ lie on the same projective line $\C P^1 \subset \C P^2$, as illustrated in Figure $\ref{fig:Y}$. 

More precisely, let $\D'\subset M_\R$ be the moment polytope in $\eqref{eq:polytope}$, corresponding to 
the toric surface $X:=\Bl_{(p_1,p_2,p_3)}(\C P^2)$ constructed in Section \ref{sec:3pts}. 
Choose the torus-fixed point $p_4\in X$ corresponding to the vertex $\bs p_4=(0,2)$ of $\D'$.
Let $Z$ denote the blow-up of $X$ at $p_4$, namely, $Z:=\Bl_{(p_4)}(X)$. 
Equivalently, we have the sequence of blow-ups
\[
Z=\Bl_{(p_4)}(X)\dasharrow X \dasharrow \C P^2.
\]
As shown in Figure \ref{fig:Y}, the corresponding polytope $\D''$ of the toric manifold $Z$ is obtained by chopping off the vertex $\bs p_4$ of $\D'$.
We therefore consider the decomposition $\D'=\D''\cup D_1$, where
\begin{align*}
\D''&=\conv\Set{\begin{pmatrix}
1 \\ 0
\end{pmatrix}, \begin{pmatrix}
2 \\ 0
\end{pmatrix},
\begin{pmatrix}
2 \\ 1
\end{pmatrix},
\begin{pmatrix}
1 \\ 2
\end{pmatrix},
\begin{pmatrix}
1/2 \\ 2
\end{pmatrix},
\begin{pmatrix}
0 \\ 3/2
\end{pmatrix},
\begin{pmatrix}
0 \\ 1
\end{pmatrix}
}, \quad \text{and} \\
D_1&=\conv\Set{\begin{pmatrix}
0 \\ 2
\end{pmatrix}, \begin{pmatrix}
0 \\ 3/2
\end{pmatrix}, 
\begin{pmatrix}
1/2 \\ 2
\end{pmatrix}}.
\end{align*}
\begin{center}
\begin{figure}[h!]
\begin{tikzpicture}[x=2.5cm,y=2.5cm]
\draw[thick,dotted] (-1.5,-1) -- (2.5,-1);
\draw[thick,dotted] (-1.0,0) -- (1.0,0);
\draw[thick] (-1.0,1) -- (0,1);
\draw[thick] (-1,1) -- (-1,0.5);
\draw[thick,dotted]  (-1,-1.5) -- (-1,2.5);
\draw[thick,dotted] (0,-1.0) -- (0,1.0);
\draw[black] (1,-1.0) -- (1,0);
\draw[thick,dotted] (2.5,-1.5) -- (-1.5,2.5);
\draw[thick,dotted,red] (-1.5,0) -- (0,1.5);
\draw[thick,red] (-1.0,0.0) -- (0.0,-1.0);
\draw[thick,red] (0,-1) -- (1,-1.0);
\draw[thick,red] (1,-1) -- (1,0);
\draw[thick,red] (1,0) -- (0,1);
\draw[thick,red] (0,1) -- (-0.5,1);
\draw[thick,red] (-0.5,1) -- (-1.0,0.5);
\draw[thick,red] (-1.0,0.5) -- (-1.0,0);
\draw[fill=blue] (-1,-1) circle (2pt) ;
\draw[fill=blue] (-1,2) circle (2pt) ;
\draw[fill=blue] (2,-1) circle (2pt) ;
\draw[fill=red] (-0.5,1) circle (2pt) ;
\draw[fill=red] (-1,0.5) circle (2pt) ;
\node at (0.2,-0.15) {$(1,1)$};
\node at (-0.3,0.25) {$\Delta''$};
\node at (-1.2,0.6) {$(0,\frac{3}{2})$};
\node at (-0.4,0.8) {$(\frac{1}{2},2)$};
\node at (-0.85,0.85) {$D_1$};
\node at (-1,-1.2) {$\boldsymbol p_1=(0,0)$};
\node at (2,-1.2) {$\boldsymbol p_2=(3,0)$};
\node at (-1,2.2) {$\boldsymbol p_3=(0,3)$};
\end{tikzpicture}
\caption{The decomposition of the polygon.}\label{fig:Y}
\end{figure}
\end{center}
To ensure that $k\D''$ a lattice polygon, we must take $k=2$.
Then the matrix defined in $\eqref{eq:Matrix}$ becomes $A_1=\begin{pmatrix}
    0 & 1 \\
    -1 &0
\end{pmatrix}$. Moreover, 
\begin{align}
\begin{split}\label{eq:BlowUpData}
M&=m_1=1, \quad  \widetilde M=m_1^2=1, \quad E_{\D'}(i)=3i^2+3i+1, \\
 A&=\#\bigl(\p (2\D')\cap \Z^2 \bigr)-M=11, \qquad B=2\Vol(2\D')-\widetilde M=23,\\
 \int_{\D'} \bs x\, dv&=\left(3, \, 3 \right),\qquad
 \bs s_{\D'}(i)= 
\begin{pmatrix}
3 \\ 3
\end{pmatrix}i^2
+
\begin{pmatrix}
3 \\ 3
\end{pmatrix}i+
\begin{pmatrix}
1 \\ 1
\end{pmatrix}, \\
\bs s_{2\D'}(i)&=2\bs s_{\D'}(2i)=
\begin{pmatrix}
24 \\ 24
\end{pmatrix}i^2
+
\begin{pmatrix}
12 \\ 12
\end{pmatrix}i+
\begin{pmatrix}
2 \\ 2
\end{pmatrix}.
\end{split}
\end{align}
Substituting these values into $\eqref{eq:DF1}$ and $\eqref{eq:DF2}$, we obtain
\begin{align*}
\mathcal{DF}_{2\D', 1}&=\frac{1}{12}\begin{pmatrix}
1 \\ -1
\end{pmatrix}
+\frac{2}{4}\left\{11
\begin{pmatrix}
0 \\ 2
\end{pmatrix}-23
\begin{pmatrix}
0 \\ 2
\end{pmatrix}
\right\}+\frac{1}{2}\int_{2\D'}\bs x\,dv-\frac{1}{4}\int_{\p(2\D')}\bs x\,d\sigma \\
&=\frac{83}{12}
\begin{pmatrix}
1 \\ -1
\end{pmatrix}, \qquad  \\
\mathcal{DF}_{2\D', 2}&=\frac{23}{12}
\begin{pmatrix}
    \alpha_1+\gamma_1 \\
    \beta_1+\delta_1
\end{pmatrix}+\frac{1}{6}
\begin{pmatrix}
    \alpha_1+\gamma_1 \\
    \beta_1+\delta_1
\end{pmatrix}
+\frac{2}{2}\bs p_4-\frac{1}{2}\bs s_{2\D'}\\
&=\frac{23}{12}\begin{pmatrix}
1 \\ -1
\end{pmatrix}+\frac{1}{6}
\begin{pmatrix}
1 \\ -1
\end{pmatrix}+
\begin{pmatrix}
0 \\ 2
\end{pmatrix}
-\begin{pmatrix}
1 \\ 1
\end{pmatrix}=
\frac{13}{12}\begin{pmatrix}
1 \\ -1
\end{pmatrix}.
\end{align*}
Therefore,
\begin{equation}\label{eq:ChowWt4pts}
\Chow_{2\D^{''}}(\bs x;i)=\Chow_{2\D'}(\bs x; i)+\frac{83}{12}\begin{pmatrix}
1 \\ -1
\end{pmatrix}i+\frac{13}{12}\begin{pmatrix}
1 \\ -1
\end{pmatrix}
\end{equation}
be $\eqref{eq:BlowUp}$. Consequently, $(Z,\mathcal O_Z(1))$ is Chow \emph{unstable}.

Note that the coefficient of $i$ in $\eqref{eq:ChowWt4pts}$ and the constant term are proportional,  since 
\[
\Chow_{2\D'}(\bs x; i)=\begin{pmatrix}
0 \\ 0
\end{pmatrix}.
\]
In other words, 
\[
\dim \mathrm{Span}_\R\Set{\mathrm{Coeff}\bigl(\Chow_{2\D^{''}}(\bs x; i), i \bigr), \mathrm{Coeff}\bigl(\Chow_{2\D^{''}}(\bs x; i), \mathrm{const} \bigr)}=1.
\]
This phenomenon can be explained as follows. Let $X_\Sigma$ be an $n$-dimensional smooth projective toric variety with associated complete fan 
$\Sigma$ in $N_\R$.
Denote by $\mathcal W(X_\Sigma)$ the Weyl group of $X_\Sigma$. By \cite[Proposition $3.1$]{BS99}, we have
\begin{equation}\label{eq:WeylGr}
\mathcal{W}(X_\Sigma)\cong \Set{\gamma \in \GL_n(\Z)| \gamma(\Sigma)=\Sigma}.
\end{equation}
In our situation, the polygon $\D^{''}$ has normal fan $\Sigma''$ in $N_\R\cong \R^2$ whose one-dimensional cones are $\Sigma''(1)=\set{\rho_1, \dots, \rho_7}$,
with primitive ray generators
\[
\pm \bs e_1, \quad  \pm \bs e_2, \quad \pm (\bs e_1+\bs e_2), \quad \bs e_1-\bs e_2, 
\]
where $\bs e_1, \bs e_2$ denote the standard basis vectors of $N_\R$. 
Hence $\Sigma''$ is invariant under the unimodular transformation
$\begin{pmatrix}
0 & -1 \\
-1 & 0
\end{pmatrix}$, which corresponds to the map $(x_1, x_2)\mapsto (-x_2,-x_1)$. 
Therefore,
\begin{equation*}
\mathcal W(Z) \cong \Braket{\begin{pmatrix}
0 & -1 \\
-1 & 0
\end{pmatrix}}\cong \Z_2 
\end{equation*} 
by $\eqref{eq:WeylGr}$.

Let $N_\R^{\mathcal W(Z)}\subset N_\R$ be the $\mathcal W(Z)$-invariant subspace. Then  
$\dim N_\R^{\mathcal W(Z)}=2-1=1$.
Since the coefficients of $\Chow_{2\D^{''}}(\bs x;i)$ are invariant under $\mathcal W(Z)$-action, we obtain
\[
\dim \mathrm{Span}_\R\Set{\mathrm{Coeff}\bigl(\Chow_{2\D^{''}}(\bs x; i) \bigr)}=\dim N_\R^{\mathcal W(Z)}=1.
\]

\subsection{Blow-ups of $\C P^2$ at five and six distinct points}\label{sec:6pts}
Let $Z$ be the toric surface constructed in Section \ref{sec:4pts}. We next choose the torus-fixed point $p_5\in Z$ corresponding to the vertex $\bs p_5=(2,0)$ of the polygon $\D''$. Let $Z_1$ denote the blow-up of $Z$ at $p_5$, namely, $Z_1:=\mathrm{Bl}_{(p_5)}(Z)$. Equivalently, we have the sequence of blow-ups
\begin{equation}\label{eq:SeqZ1}
Z_1=\mathrm{Bl}_{(p_5)}(Z)\dasharrow Z \dasharrow X \dasharrow \C P^2.
\end{equation}
The corresponding decomposition of the associated polygon is 
\begin{equation}\label{eq:decomp}
\D'=\D_1\cup \smallcup_{a=1}^2 D_a=\D^{''}\cup D_1,
\end{equation} 
where 
\begin{align*}
\D_1&=\conv\Set{\begin{pmatrix}
1 \\ 0
\end{pmatrix}, 
\begin{pmatrix}
3/2 \\ 0
\end{pmatrix},
\begin{pmatrix}
2 \\ 1/2
\end{pmatrix},
\begin{pmatrix}
2 \\ 1
\end{pmatrix},
\begin{pmatrix}
1 \\ 2
\end{pmatrix},
\begin{pmatrix}
1/2 \\ 2
\end{pmatrix},
\begin{pmatrix}
0 \\ 3/2
\end{pmatrix},
\begin{pmatrix}
0 \\ 1
\end{pmatrix}
}, \quad \text{and}  \\
D_2&=\conv\Set{\begin{pmatrix}
2 \\ 0
\end{pmatrix}, \begin{pmatrix}
2 \\ 1/2
\end{pmatrix}, 
\begin{pmatrix}
3/2 \\ 0
\end{pmatrix}}.
\end{align*}
In this case, $k=2$ and the matrix defined in $\eqref{eq:Matrix}$ is
$A_2=\begin{pmatrix}
0 & -1 \\ 
1 & 0
\end{pmatrix}$. Proceeding as in $\eqref{eq:BlowUpData}$, we obtain
\begin{align*}
M&=\sum_{j=1}^2 m_j=2, \qquad  \widetilde M=\sum_{j=1}^2 m_j^2=2, \quad \text{and}  \\
 A&=\#\bigl(\p (2\D')\cap \Z^2 \bigr)-M=10, \qquad B=2\Vol(2\D')-\widetilde M=22.
\end{align*}
Therefore by $\eqref{eq:BlowUp}$, the Chow weight of $(Z_1, \mathcal O_{Z_1}(1))$ is given by 
\[
\Chow_{\D_1}(\bs x;i)=\Chow_{\D'}(\bs x;i)+i \mathcal{DF}_{2\D',1}+\mathcal{DF}_{2\D',2},
\]
 where $\mathcal{DF}_{2\D',1}$ amd $\mathcal{DF}_{2\D',2}$ are given by 
\begin{align*}
\mathcal{DF}_{2\D',1}&=\frac{10}{12}\sum_{j=1}^2
\begin{pmatrix}
\alpha_j+\gamma_j \\
\beta_j+\delta_j
\end{pmatrix}+\frac{2}{4}\bigl(10\sum_{\ell=4}^5 \bs p_\ell -22\sum_{\ell=4}^5 \bs p_\ell \bigr)
+\frac{2}{2}\begin{pmatrix}
24 \\24
\end{pmatrix}-\frac{2}{4}\begin{pmatrix}
24 \\24
\end{pmatrix}=\begin{pmatrix}
0 \\0
\end{pmatrix}, \quad \\
\mathcal{DF}_{2\D',2}&=\frac{22}{12}\sum_{j=1}^2
\begin{pmatrix}
\alpha_j+\gamma_j \\
\beta_j+\delta_j
\end{pmatrix}+\frac{1}{6}\sum_{j=1}^2
\begin{pmatrix}
\alpha_j+\gamma_j \\
\beta_j+\delta_j
\end{pmatrix}+ \frac{2}{2}\sum_{\ell=4}^5 \bs p_\ell-\frac{2}{2}\bs s_{2\D'}\\
&=\sum_{\ell=4}^5 \bs p_\ell-\bs s_{2\D'}=\left\{\begin{pmatrix}
0 \\2
\end{pmatrix} +\begin{pmatrix}
2 \\0
\end{pmatrix}\right\}-\begin{pmatrix}
2 \\2
\end{pmatrix}=\begin{pmatrix}
0 \\0
\end{pmatrix},
\end{align*}
respectively. Since $\Chow_{\D_1}(\bs x;i)\equiv \bs 0$, 
we conclude that $(Z_1,\mathcal O_{Z_1}(1))$ is {\emph{asymptotically}} Chow polystable.

Finally, let $p_6 \in Z_1$ be the torus-fixed point corresponding to the vertex $\bs p_6=(1,2)$ of $\Delta_1$.
The induced toric blow-up is  $Z_2:=\mathrm{Bl}_{(p_6)}(Z_1)$, yielding the decomposition 
\[
\D_1=\D_2\cup D_3,
\] 
where
 \begin{align*}
\D_2&=\conv\Set{\begin{pmatrix}
1 \\ 0
\end{pmatrix}, 
\begin{pmatrix}
3/2 \\ 0
\end{pmatrix},
\begin{pmatrix}
2 \\ 1/2
\end{pmatrix},
\begin{pmatrix}
2 \\ 1
\end{pmatrix},
\begin{pmatrix}
5/4 \\ 7/4
\end{pmatrix},
\begin{pmatrix}
3/4 \\ 2
\end{pmatrix},
\begin{pmatrix}
1/2 \\ 2
\end{pmatrix},
\begin{pmatrix}
0 \\ 3/2
\end{pmatrix},
\begin{pmatrix}
0 \\ 1
\end{pmatrix}
} \qquad \text{and} \\
D_3&=\conv\Set{\begin{pmatrix}
1 \\ 2
\end{pmatrix}, \begin{pmatrix}
3/4 \\ 2
\end{pmatrix}, 
\begin{pmatrix}
5/4 \\ 7/4
\end{pmatrix}}.
\end{align*}
(see Figure $\ref{fig:Z2}$.)
The minimal positive integer $k$ such that $k\D_2$ is a lattice polygon is $k=4$. 
Furthermore, the matrix $A_3$ in $\eqref{eq:Matrix}$ and each value of invariants are 
\begin{align*}
A_3&=\begin{pmatrix}
    -1 & 1 \\
    0 &-1
\end{pmatrix}, 
\qquad M=m_3=1, \qquad  \widetilde M= m_3^2=1,   \qquad \\
 A&=\#\bigl(\p (4\D_1)\cap \Z^2 \bigr)-M=20-1=19,  \qquad 
 B=2\Vol(4\D_1)-\widetilde M=2\cdot 44-1=87.
\end{align*}
\begin{center}
\begin{figure}
\begin{tikzpicture}[x=5.0cm,y=5.0cm]
\draw[thick,dotted] (-1.0,0) -- (1.0,0);
\draw[thick,dotted] (0,-1.0) -- (0,1.0);
\draw[thick,dotted,red] (-0.75,1.25) -- (0.75,0.5);
\draw[thick,red] (-0.25,1) -- (0.25,0.75);
\draw[thick,red] (-1.0,0.0) -- (0.0,-1.0);
\draw[thick,red] (0,-1) -- (0.5,-1);
\draw[thick,red] (1,-0.5) -- (1,0);
\draw[thick,red] (1,0) -- (0,1);
\draw[thick,red] (-0.25,1) -- (-0.5,1);
\draw[thick] (0,1) -- (-0.25,1);
\draw[thick] (0,1) -- (0.25,0.75);
\draw[thick,red] (-0.5,1) -- (-1.0,0.5);
\draw[thick,red] (-1.0,0.5) -- (-1.0,0);
\draw[thick,red] (0.5,-1) -- (1.0,-0.5);
\draw[fill=blue] (0,1) circle (2pt) ;
\draw[fill=red] (-0.25,1) circle (2pt) ;
\draw[fill=red] (0.25,0.75) circle (2pt) ;
\node at (0.2,-0.15) {$(1,1)$};
\node at (-0.3,0.25) {$\Delta_2$};
\node at (-1.2,0.6) {$(0,\frac{3}{2})$};
\node at (-1.2,0) {$(0,1)$};
\node at (1.2,0) {$(2,1)$};
\node at (1.2,-0.5) {$(2,\frac{1}{2})$};
\node at (0.7,-0.95) {$(\frac{3}{2},0)$};
\node at (-0.65,1.0) {$(\frac{1}{2},2)$};
\node at (-0.25,0.9) {$(\frac{3}{4},2)$};
\node at (0.25,0.67) {$(\frac{5}{4},\frac{7}{4})$};
\node at (0,0.93) {$D_3$};
\node at (0.25,1.1) {$\boldsymbol p_6=(1,2)$};
\end{tikzpicture}
\caption{The toric blow-up of $\C P^2$ at six torus-fixed points.
}\label{fig:Z2}
\end{figure}
\end{center}
Direct computation yields
\begin{align*}
    E_{4\D_2}(i)=44i^2+10i+1, \quad \int_{4\D_2}\bs x\, dv=
    \begin{pmatrix}
        176 \\ 176
    \end{pmatrix},\quad
    \bs s_{4\D_2}(1)=\begin{pmatrix}
        220 \\ 220
    \end{pmatrix}, \quad
    \bs s_{4\D_2}(2)=\begin{pmatrix}
        1576 \\ 1576
    \end{pmatrix}.
\end{align*}
Hence, by $\eqref{eq:Ehr_sum}$, 
\[
\bs s_{4\D_2}(i)=
176\begin{pmatrix}
        1 \\ 1
    \end{pmatrix}i^2+828\begin{pmatrix}
        1 \\ 1
    \end{pmatrix}i
    -784176\begin{pmatrix}
        1 \\ 1
    \end{pmatrix}.
\]
Consequently,
\begin{align*}
\mathcal{DF}_{4\D_2,1}&=\frac{19}{12}
\begin{pmatrix}
\alpha_3+\gamma_3 \\
\beta_3+\delta_3
\end{pmatrix}+\frac{4}{4}\bigl(19 \bs p_6 -87 \bs p_6 \bigr)
+\frac{1}{2}\begin{pmatrix}
176 \\176
\end{pmatrix}-\frac{1}{4}\begin{pmatrix}
1656 \\1656
\end{pmatrix}\\
&=\frac{19}{12}
\begin{pmatrix}
0 \\
-1
\end{pmatrix}
-68\begin{pmatrix}
1 \\
2
\end{pmatrix}+(88-414)\begin{pmatrix}
1 \\
1
\end{pmatrix}
=\begin{pmatrix}
-394 \\-\frac{4747}{12}
\end{pmatrix},     \quad \text{and} \\
\mathcal{DF}_{4\D_2,2}&=\frac{87}{12}
\begin{pmatrix}
0 \\
-1
\end{pmatrix}+\frac{1}{6}\begin{pmatrix}
0 \\
-1
\end{pmatrix}+2\begin{pmatrix}
1 \\
2
\end{pmatrix}-392\begin{pmatrix}
1 \\
1
\end{pmatrix}=
\begin{pmatrix}
-390 \\
-\frac{4745}{12}
\end{pmatrix}.
\end{align*}
As a consequence, we obtain $\dim \mathrm{Span}_\R\Set{\mathrm{Coeff}(\Chow_{4\D_2}(\bs x;i))}=2$ and
\begin{equation}\label{eq:ChowWt6pts} 
\Chow_{4\D_2}(\bs x;i)=-\begin{pmatrix}
394 \\
\frac{4747}{12}
\end{pmatrix}i-\begin{pmatrix}
390 \\
\frac{4745}{12}
\end{pmatrix}.
\end{equation}
Therefore, $(X_{\D_2}, \mathcal O_{X_{\D_2}}(4))$ is Chow unstable.

On the other hand, suppose that $\Lambda=\set{p_1,p_2, \dots, p_6}\subset \C P^2$ is a collection of six points in general position; namely, no three points are collinear and no six points lie on a conic. Then
\[
\widetilde X:=\mathrm{Bl}_\Lambda(\C P^2)\dasharrow \C P^2 
\]
is the del Pezzo surface of degree three. In particular, $\widetilde X$ admits a K\"ahler-Einstein metric in the class $c_1(X)$ by a theorem of Tian; 
see, \cite{Ti90} and \cite[Theorem $1.4$]{Ts12}.
Moreover, it is known that the Lie algebra $\mathfrak h(\widetilde X)$ of holomorphic vector fields on the blow-up of $\C P^2$ at four 
or more general points is trivial: $\mathfrak h(\widetilde X)=\set{0}$. Hence, by Donaldson's theorem \cite{Don01},
$(\widetilde X, -K_{\widetilde X})$ is asymptotically Chow stable. 

Therefore, the computation in $\eqref{eq:ChowWt6pts}$ reveals a striking contrast between the Chow weights of toric blow-ups and those arising from blow-ups at points in general position.

\subsection{The blow-up at a point on a Hirzebruch surface}\label{sec:1stStep}
We conclude this section with the following observation.
If the polygon $\D\subset \R^2$ is a rectangle, then $\Chow_\D(\bs x;i)=\bs 0$. 
Thus, without loss of generality, we may assume that a Delzant polygon $\D\subset \R^2$ with $\#\mathcal V(\D)=4$ has vertices
\[
\bs p_1=\begin{pmatrix}
0 \\ 0
\end{pmatrix}, \quad
\bs p_2=\begin{pmatrix}
0 \\ a
\end{pmatrix}, \quad
\bs p_3=\begin{pmatrix}
b \\ a
\end{pmatrix}, \quad
\bs p_4=\begin{pmatrix}
b+an \\ 0
\end{pmatrix}, 
\]
where $a$, $b$, $n$ are positive integers. Let $\D_n(a,b):=\conv\set{\bs p_1, \bs p_2, \bs p_3, \bs p_4}$.
We remark that every two-dimensional Delzant polygon with four vertices, except rectangles, can be written in the form $\D_n(a,b)$ up to parallel translations and $SL_2(\Z)$-transformations. 

For later use, we compute the following quantities: 
\begin{proposition}\label{prop:FundVal}
We have
\begin{align*}
\Vol(\D_n(a,b))&=ab+\frac{1}{2}a^2n, \\
E_{\D_n(a,b)}(i)&=\left(ab+\frac{1}{2}a^2n\right)i^2+\left(a+b+\frac{1}{2}an\right)i+1,\\
\int_{\D_n(a,b)}\bs x\, dv&=\frac{a}{6}
\begin{pmatrix}
a^2n^2+3abn+3b^2 \\
a(an+3b)
\end{pmatrix}, \qquad \text{and}\\
\bs s_{\D_n(a,b)}(i)&=\frac{a}{6}
\begin{pmatrix}
a^2n^2+3abn+3b^2 \\
a(an+3b)
\end{pmatrix}i^2+\frac{1}{4}
\begin{pmatrix}
a^2n^2+a^n+2abn+2ab+2b^2 \\
2a(a+b)
\end{pmatrix}i \\
&+\frac{1}{12}\begin{pmatrix}
an^2+3an+6b \\
2a(3-n)
\end{pmatrix}.
\end{align*}
\end{proposition}
\begin{proof}
All calculations are straightforward and are left to the reader.   
\end{proof}
Substituting the quantities in Proposition \ref{prop:FundVal} into $\eqref{def:ChowWt}$, we obtain
\[
\Chow_{\D_n(a,b)}(\bs x; i)=\frac{1}{24}a^2n(ai+1)(an-a+2b)
\begin{pmatrix}
n \\ -2
\end{pmatrix}.
\]
Next, we compute the Chow weight of the manifold obtained by successively blowing up the four torus-fixed points $\bs p_1$, $\bs p_2$, $\bs p_3$ and $\bs p_4$.
For these blow-ups, we have
\begin{align*}
\begin{pmatrix}
\alpha_1+\gamma_1 \\
\beta_1+\delta_1
\end{pmatrix}&=
\begin{pmatrix}
1 \\ 1
\end{pmatrix}, \quad
\begin{pmatrix}
\alpha_2+\gamma_2 \\
\beta_2+\delta_2
\end{pmatrix}=
\begin{pmatrix}
1 \\ -1
\end{pmatrix}, \quad
\begin{pmatrix}
\alpha_3+\gamma_3 \\
\beta_3+\delta_3
\end{pmatrix}=
\begin{pmatrix}
n-1 \\ -1
\end{pmatrix}, \quad
\begin{pmatrix}
\alpha_4+\gamma_4 \\
\beta_4+\delta_4
\end{pmatrix}=
\begin{pmatrix}
-1-n \\ 1
\end{pmatrix},
\end{align*} 
respectively. Thus, 
\begin{align}\label{eq:EachValue1}
\begin{split}
\sum_{a=1}^4m_a^3\begin{pmatrix}
\alpha_a+\gamma_a \\
\beta_a+\delta_a
\end{pmatrix}&=
\begin{pmatrix}
m_1^3+m_2^3+(n-1)m_3^3-(n+1)m_4^3 \\
m_1^3-m_2^3-m_3^3+m_4^3
\end{pmatrix}, \\
\sum_{a=1}^4m_a\begin{pmatrix}
\alpha_a+\gamma_a \\
\beta_a+\delta_a
\end{pmatrix}&=
\begin{pmatrix}
m_1+m_2+(n-1)m_3-(n+1)m_4 \\
m_1-m_2-m_3+m_4
\end{pmatrix}, \\
\sum_{a=1}^4m_a^2\bs p_a=&
\begin{pmatrix}
bm_3^2+(b+an)m_4^2 \\
a(m_2^2+m_3^2)
\end{pmatrix}, \quad \text{and} \quad
\sum_{a=1}^4m_a\bs p_a=
\begin{pmatrix}
bm_3+(b+an)m_4 \\
a(m_2+m_3)
\end{pmatrix}.
\end{split}
\end{align}
Since $k\D_n(a,b)=\D_n(ka,kb)$, we further obtain
\[
\Chow_{k\D_n(a,b)}(\bs x; i)=\frac{k^3}{24}a^2n(aki+1)(an-a+2b)
\begin{pmatrix}
n \\ -2
\end{pmatrix},
\]
together with
\begin{align}
\begin{split}\label{eq:EachValue2}
A&=k(2a+an+2b)-M, \quad B=ak^2(an+2b)-\widetilde M, \\
\int_{k\D_n(a,b)}\bs x\, dv&=\frac{ak^3}{6}
\begin{pmatrix}
a^2n^2+3abn+3b^2 \\
a(an+3b)
\end{pmatrix}, \\
\int_{\p(k\D_n(a,b))}\bs x\, d\sigma&= \frac{k^2}{2}
\begin{pmatrix}
a^2n^2+a^2n+2abn+2ab+2b^2 \\
2a(a+b)
\end{pmatrix}, \\
\bs s_{k\D_n(a,b)}&=\frac{k}{12}
\begin{pmatrix}
an^2+3an+6b \\
2a(3-n)
\end{pmatrix}.
\end{split}
\end{align}
Substituting $\eqref{eq:EachValue1}$ and $\eqref{eq:EachValue2}$ into $\eqref{eq:DF1}$ and $\eqref{eq:DF2}$, 
we obtain an explicit formula for 
\[
\Chow_{k\D_n'(a,b)}(\bs x;i)
\] 
via $\eqref{eq:BlowUp}$.

\vskip 7pt

\noindent{\bfseries Conflict of interest:} The authors have no competing interests to declare that are relevant to the content of this article.

\vskip 7pt

\noindent{\bfseries Data Availability Statement:} Data sharing not applicable to this article as no datasets were generated or analysed during the current study.


\end{document}